\documentclass[12pt]{amsart}
\usepackage{epsfig}
\usepackage{graphics}
\numberwithin{equation}{section}

   \newtheorem{thm}{Theorem}[section]
   \newtheorem{cor}[thm]{Corollary}
   \newtheorem{prop}[thm]{Proposition}
   \newtheorem{lemma}[thm]{Lemma}
   \newtheorem{claim}[thm]{Claim}

\theoremstyle{definition}
   \newtheorem{dfn}[thm]{Definition}
   
\theoremstyle{remark}
   \newtheorem{rem}[thm]{Remark}

\newcommand{\m}{\underline{m}}
\newcommand{\C}{\mathbb C}
\newcommand{\fM}{\mathfrak m}
\newcommand{\fL}{\mathfrak l}

\newcommand{\Z}{\mathbb Z}

\newcommand{\ve}{\varepsilon}
\newcommand{\bd}{\partial}

\begin{document}
\title[Compatible contact structures of fibered graph multilinks]{Compatible contact structures of fibered positively-twisted graph multilinks in the 3-sphere}
\author{Masaharu Ishikawa}
\footnote[0]{This work is supported by MEXT, Grant-in-Aid for Young Scientists (B) (No. 22740032).}
\address{Mathematical Institute, Tohoku University, Sendai, 980-8578, Japan}
\email{ishikawa@math.tohoku.ac.jp}

\begin{abstract}
We study compatible contact structures of fibered, positively-twisted graph 
multilinks in $S^3$ and prove that the contact structure of such a multilink
is tight if and only if the orientations of its link components are
all consistent with or all opposite to the orientation of the fibers of the
Seifert fibrations of that graph multilink.
As a corollary, we show that the compatible contact structures of
the Milnor fibrations of real analytic germs of the form $(f\bar g,O)$
are always overtwisted.
\end{abstract}

\maketitle

\section{Introduction}

A contact structure on a closed, oriented, smooth $3$-manifold $M$ is
the kernel of a $1$-form $\alpha$ on $M$ satisfying $\alpha \land d\alpha\ne 0$
everywhere. In this paper, we only consider positive contact forms
i.e., contact forms $\alpha$ satisfying $\alpha\land d\alpha>0$.
The idea of contact structures compatible with fibered links in $M$
was first introduced by W.P.~Thurston and H.~Winkelnkemper in~\cite{tw}
and developed by E.~Giroux in~\cite{giroux}.
In the previous work~\cite{ishikawa2}, 
the compatible contact structures of fibered Seifert multilinks
in Seifert fibered homology $3$-spheres $\Sigma(a_1,a_2,\cdots,a_k)$ were studied.
Here $a_i$'s are the denominators of the Seifert invariants.
Especially, we determined their tightness in case $a_1a_2\cdots a_k>0$.
The $3$-sphere with positive Hopf fibration is
a typical example of Seifert fibered homology $3$-spheres satisfying this inequality.

A graph multilink is obtained from Seifert multilinks by iterating 
a certain gluing operation, called a {\it splicing}.
We focus on fibered graph multilinks obtained as a splice of Seifert multilinks 
in homology $3$-spheres with $a_1a_2\cdots a_k> 0$,
which we call {\it positively-twisted graph multilinks} in homology $3$-spheres.
For convenience, we may assume that the denominators of the
Seifert invariants of the Seifert fibered homology
$3$-spheres constituting the graph multilink are all positive. This is always possible
as mentioned in~\cite[Proposition~7.3]{en}. In this setting,
we say that the orientation of a graph multilink is {\it canonical}
if the multiplicities of its link components are either all positive or all negative.

In this paper, we determine the tightness of positively-twisted graph multilinks in $S^3$.

\begin{thm}\label{thm01}
The compatible contact structure of a fibered, positively-twisted graph multilink in $S^3$ 
is tight if and only if its orientation is canonical.
\end{thm}

A typical example of positively-twisted graph multilinks in $S^3$
is an oriented link obtained from a trivial knot in $S^3$ by iterating
``positive'' cablings. Here a ``positive'' cabling means that
the cabling coefficients are positive with respect to the framing of the
Seifert surface of the link component for the cabling.

The situation in Theorem~\ref{thm01} occurs when we consider
the Milnor fibration of a real analytic germ of the form $(f\bar g,O)$,
where $f,g:(\C^2,O)\to (\C,0)$ are holomorphic germs at the origin $O\in\C^2$
and $\bar g$ represents the complex conjugation of $g$.
In~\cite{pichon, ps}, it is proved that
\[
   \frac{f\bar g}{|f\bar g|}:S_\ve\setminus \{fg=0\}\to S^1
\]
is a locally trivial fibration in most cases, called the {\it Milnor fibration}
of $(f\bar g,O)$, where $S_\ve$ is the $3$-sphere centered at $O\in\C^2$ with
sufficiently small radius $\ve>0$.
The next result answers a question of A.~Pichon asked in her talk in Luminy, May, 2006
(cf.~\cite{ishikawa}).

\begin{cor}\label{cor03}
Suppose that 
   $\frac{f\bar g}{|f\bar g|}:S_\ve\setminus \{fg=0\}\to S^1$ is a locally trivial
fibration. Then its compatible contact structure is overtwisted.
\end{cor}

This paper is organized as follows.
In Section~2, we fix the notations of Seifert fibered homology $3$-spheres, 
Seifert multilinks, and graph multilinks, following the book~\cite{en}.
The notion of their compatible contact structures is also introduced in this section.
The next two sections are devoted to preparations for the proof of Theorem~\ref{thm01}.
In Section~3, we give a compatible contact structure of a fibered, positively-twisted 
graph multilink in a special case. We then give another compatible contact structure
in Section~4 based on the Thurston-Winkelnkemper's construction in~\cite{tw}.
These compatible contact structures will be used to prove Theorem~\ref{thm01}
in Section~5. Corollary~\ref{cor03} will also be proved in this section.


\section{Preliminaries}

In the following, int$X$ and $\bd X$ represent the interior and the boundary of a
topological space $X$ respectively.

\subsection{Notation of Seifert fibered homology $3$-spheres}

We follow the notation used in~\cite{ishikawa2}, which originally appears in~\cite{en}.

Let $\mathcal S=S^2\setminus\text{int}(D_1^2\cup\cdots\cup D_k^2)$
be a 2-sphere with $k$ holes and make
an oriented, closed, smooth $3$-manifold $\Sigma$ from $\mathcal S\times S^1$ by 
gluing solid tori $(D^2\times S^1)_1,\cdots,(D^2\times S^1)_k$
along the boundary $\bd (\mathcal S\times S^1)$ in such a way that 
$a_iQ_i+b_iH$ is null-homologous in $(D^2\times S^1)_i$,
where 
\[
\begin{split}
   Q_i&=(-\bd \mathcal S^{\text{\rm sec}})\cap (D^2\times S^1)_i \\
   H&=\text{typical oriented fiber of $\pi$ in $\bd (D^2\times S^1)_i$},
\end{split}
\]
with $\mathcal S^{\text{\rm sec}}$ a section of $\pi:\mathcal S\times S^1\to \mathcal S$ and
$(a_i,b_i)\in\Z^2\setminus\{(0,0)\}$
are chosen such that $\sum_{i=1}^kb_ia_1\cdots a_{i-1} a_{i+1}\cdots a_k=1$.
The obtained $3$-manifold does not depend on the ambiguity of the choice of $b_i$'s,
so we may denote it simply as $\Sigma=\Sigma(a_1,\cdots,a_k)$.
The core curve $S_i$ of each solid torus $(D^2\times S^1)_i$
is a fiber of the Seifert fibration after the gluings.
We assign to $S_i$ an orientation
in such a way that the linking number of $S_i$ and $a_i Q_i+b_i H$ equals $1$.
This orientation is called the {\it working orientation}.

Let $(\fM_i,\fL_i)$ be the preferred meridian-longitude pair of the link complement 
$\Sigma\setminus S_i$ chosen such that the orientation of the longitude $\fL_i$
agrees with the working orientation of $S_i$.
In this setting, $(\fM_i,\fL_i)$ and $(Q_i,H)$ are related by the following equations, 
see~\cite[Lemma~7.5]{en}:
\begin{equation}\label{eqab}
   \begin{pmatrix} \fM_i \\ \fL_i \end{pmatrix}
   =
   \begin{pmatrix} a_i & b_i \\ -\sigma_i & \delta_i \end{pmatrix}
   \begin{pmatrix} Q_i \\ H \end{pmatrix}\quad\text{and}\quad
   \begin{pmatrix} Q_i \\ H \end{pmatrix}
   =
   \begin{pmatrix} \delta_i & -b_i \\ \sigma_i & a_i \end{pmatrix}
   \begin{pmatrix} \fM_i \\ \fL_i \end{pmatrix},
\end{equation}
where $\sigma_i=a_1\cdots\hat a_i\cdots a_k$ and 
$\delta_i=\sum_{i\ne j}b_ja_1\cdots \hat a_i\cdots \hat a_j\cdots a_k$. 
Note that they satisfy $a_i\delta_i+b_i\sigma_i=1$.

Set $A=a_1\cdots a_k$.
Under the assumption $A\ne 0$, the orientations of the fibers of the Seifert fibration
in $\mathcal S\times S^1\to \mathcal S$ are canonically extended 
to the fibers in $(D^2\times S^1)_i$ for each $i=1,\cdots,k$,
which we call the {\it orientation of the Seifert fibration}.

\subsection{Seifert multilinks in $\Sigma(a_1,\cdots,a_k)$}

A Seifert link $L$ is a link in $\Sigma(a_1,\cdots,a_k)$ consisting of
a finite number of fibers of the Seifert fibration.
We may choose the core curves $S_1,\cdots,S_k$ of the solid tori $(D^2\times S^1)_i$
such that $S_1\cup S_2\cup \cdots\cup S_n$ is the Seifert link $L$ for some $n\leq k$.
A Seifert multilink is a Seifert link
each of whose link components is equipped with a non-zero integer, called the {\it multiplicity}.
We denote the set of multiplicities as $\m=(m_1,\cdots,m_n)$ and
the Seifert multilink as 
\[
   (\Sigma,L(\m))=(\Sigma(a_1,\cdots,a_k), m_1S_1\cup\cdots\cup m_nS_n).
\]
We may denote it simply as $L(\m)$.

Each link component of $L(\m)$ is canonically oriented according to its multiplicity,
i.e., the orientation is defined to be consistent with the working orientation if $m_i>0$
and opposite to it if $m_i<0$.
In this paper, we allow $m_i$ to be $0$ for convenience,
which means that $S_i$ is not a component of $L(\m)$.
We may call such an $S_i$ also a link component of $L(\m)$ though it is an empty component.

\begin{dfn}
A Seifert multilink $(\Sigma,L(\m))$ is called {\it positively-twisted} (or {\it PT} for short)
if $A=a_1\cdots a_k>0$.
\end{dfn}

\begin{rem}
The notion of positivity is usually used for oriented links, as positive braids and positive links.
We here say that ``a Seifert multilink is positively-twisted'' because 
this is a notion for (multi-)links without specific orientation.
\end{rem}

\begin{dfn}
Suppose $A\ne 0$ and fix an orientation of the Seifert fibration.
We say that a link component $m_iS_i$ of a Seifert multilink $L(\m)$ with $m_i\ne 0$ is
{\it positive} (resp. {\it negative}) if its orientation is consistent with
(resp. opposite to) the orientation of the Seifert fibration.
\end{dfn}

\subsection{Splicing and graph multilinks}

Let $(\Sigma_1, L_1)$ and $(\Sigma_2, L_2)$ be links in homology $3$-spheres
$\Sigma_1$ and $\Sigma_2$ respectively. Choose a link component $S_1$ of $L_1$ and also
$S_2$ of $L_2$.
For each $i=1,2$,
let $N(S_i)$ be a compact tubular neighborhood of $S_i$ and
$(\fM_i,\fL_i)$ be a preferred meridian-longitude pair of $\Sigma_i\setminus\text{int\,}N(S_i)$.
Remark that the longitude $\fL_i$ is chosen such that it is null-homologous in 
the exterior $\Sigma_i\setminus\text{int\,}N(S_i)$.
We then glue the two exteriors
$\Sigma_1\setminus \text{int\,} N(S_1)$ and $\Sigma_2\setminus \text{int\,} N(S_2)$
in such a way that $(\fM_1,\fL_1)$ are identified with $(\fL_2,\fM_2)$ along the boundaries.
The link $(L_1\setminus S_1)\cup(L_2\setminus S_2)$ in the glued manifold 
is called the {\it splice} of $L_1$ and $L_2$ along $S_1$ and $S_2$.
Note that the glued manifold again becomes a homology $3$-sphere.

Let $\Gamma$ be a connected, simply-connected finite graph with the following decorations:
\begin{itemize}
\item
Each terminal vertex is either a boundary vertex or an arrowhead vertex.
Here the former is a usual vertex and the latter means the arrowhead of an edge $\longrightarrow$
of arrow shape.
\item
Each non-terminal vertex has sign $+$ or $-$.
We call such a vertex an {\it inner vertex}
and an edge connecting two inner vertices an {\it inner edge}.
\item 
For each inner vertex, integers $a_1,\cdots,a_k$ representing a Seifert fibered 
homology $3$-sphere $\Sigma(a_1,\cdots,a_k)$ are assigned to the roots of 
the edges connected to that vertex.
\item A non-zero integer $m_i$ is assigned to each arrowhead vertex.
\end{itemize}
Such a diagram is called a {\it splice diagram}.
In this paper, we only consider splice diagrams whose inner vertices have only $+$ signs.
So, we do not need to mind these signs.

We define a multilink from a given splice diagram $\Gamma$ as follows.
First we prepare a Seifert link $(\Sigma(a_1,\cdots,a_k), S_1\cup\cdots\cup S_n)$
 for each inner vertex with integers $a_1,\cdots,a_k$ at the roots of the adjacent edges.
If two inner vertices are connected by an edge, then we apply a splicing to
the corresponding Seifert links.
Applying splicings for all inner edges successively,
we obtain a new link $L$ in a homology $3$-sphere $\Sigma$.
We then define the multiplicity of each link component of $L$ to be
the integer assigned to the corresponding arrowhead vertex.
Note that the working orientations of the link components of $L$ are defined to be
those of the Seifert fibered homology $3$-spheres.
The multilink obtained from $\Gamma$ is called a
{\it graph multilink} and denoted as $L(\Gamma)$.

Now we consider the inverse operation of splicing. Namely, decompose
$\Sigma$ along a torus $T$ and then fill the two boundary components
by solid tori $N(S_1)$ and $N(S_2)$ with new link components $S_1$ and $S_2$
being the core curves of $N(S_1)$ and $N(S_2)$ respectively, so that we obtain two new graph links
$(\Sigma_1, L_1)$ and $(\Sigma_2,L_2)$. 
The restriction of the Seifert surface of $L(\Gamma)$ to $\Sigma_i\setminus \text{int\,}N(S_i)$ 
is canonically extended into $N(S_i)$. If the Seifert surface in $N(S_i)$ is a disjoint union 
of meridional disks, then we define the multiplicity of $S_i$ to be $0$.
Otherwise, it is defined to be the number of local leaves of the Seifert surface along $S_i$,
with sign $-$ if the orientation of $S_i$ as the boundary of the Seifert surface of $L_i$
is opposite to its working orientation.
The multiplicities of the other link components of $L_1$
and $L_2$ are defined to be those of $L(\Gamma)$. 
We denote the obtained multilinks as $(\Sigma_i, L(\Gamma_i))$,
where $\Gamma_1$ and $\Gamma_2$ are the corresponding splice diagrams,
and their splice as 
\[
    (\Sigma, L(\Gamma))
     =\left[(\Sigma_1, L(\Gamma_1)) \textstyle\frac{\;\;\;\;\;\;}{S_1\;\;\;\;\;S_2} (\Sigma_2, L(\Gamma_2))\right].
\]
Applying such a decomposition successively,
we can represent $(\Sigma, L(\Gamma))$ as a splice of several Seifert multilinks
in homology $3$-spheres each of which corresponds to an inner vertex of $\Gamma$.
In other words, there exists a set of disjoint tori $T_1,\cdots,T_\tau$ 
corresponding to the inner edges of $\Gamma$ such that
each connected component of $\Sigma\setminus \sqcup_{j=1}^\tau T_j$ 
is a part of a Seifert fibered homology $3$-sphere.
We call these connected components the {\it Seifert pieces} of $(\Sigma, L(\Gamma))$.

In this paper, we only consider the following special class of graph multilinks.

\begin{dfn}
A graph multilink is called {\it positively-twisted} (or {\it PT} for short)
if it is obtained as a splice of only PT Seifert multilinks in homology $3$-spheres.
\end{dfn}

To simplify the argument, we hereafter assume that the denominators of the Seifert invariants 
of the Seifert fibered homology
$3$-spheres before the splicing are all positive.
We can always assume this by~\cite[Proposition~7.3]{en}.
Note that, under this assumption, the orientation of the Seifert fibration coincides
with the working orientations of the link components for each Seifert multilink before
the splicing.

In~\cite[Theorem~8.1]{en}, six operations to produce equivalent splice diagrams
are introduced and a splice diagram is called {\it minimal} if there is no
equivalent splice diagram with fewer edges.
In this paper, we will only deal with splice diagrams of fibered PT graph multilinks
whose inner vertices have sign $+$ and each of whose Seifert fibered
homology $3$-spheres before the splicing has the Seifert invariants with positive denominators.
In this setting, two splice diagrams are equivalent
if they are connected by the following two operations and their inverses:
\begin{itemize}
\item[3)] Let $v$ be an inner vertex.
If an edge connected to $v$ is assigned the integer $1$ at the root and has a boundary vertex
at the other endpoint then remove the edge and the boundary vertex.
Furthermore, if the number of remaining edges connected to $v$
is $2$ then remove the inner vertex $v$ and connect the two edges so that they become
a single edge.
\item[6)] Let $v$ and $v'$ be inner vertices connected by an inner edge.
Let $a_0,a_1\cdots,a_r$ and $a'_0,a'_1,\cdots,a'_s$ be the denominators
of the Seifert invariants assigned to $v$ and $v'$ respectively
such that $a_0$ and $a'_0$ are assigned to the inner edge connecting them.
If $a_0a_0'=a_1\cdots a_ra_1'\cdots a_s'$ is satisfied
then replace the vertices $v_1$ and $v_2$ and the edge connecting them by a single inner vertex.
\end{itemize}
The numbers $3)$ and $6)$ correspond to those in~\cite[Theorem~8.1]{en}.
We say that a minimal splice diagram is of type $\leftrightarrow$ if
it consists of one edge with arrowhead vertices at both endpoints.
This will be an exceptional case as in~\cite[Theorem~11.2]{en}.

\subsection{Fibered graph multilinks and contact structures}

We first briefly recall the terminologies in $3$-dimensional contact topology.
See~\cite{geiges,os} for general references.

A {\it contact structure} on $M$ is the $2$-plane field given by
the kernel of a $1$-form $\alpha$ satisfying $\alpha\land d\alpha\ne 0$ everywhere on $M$.
In this paper, we always assume that a contact structure is positive, i.e.,
it is given as the kernel of a $1$-form $\alpha$ satisfying $\alpha\land d\alpha>0$,
called a {\it positive contact form} on $M$.
A vector field $R_\alpha$ on $M$ determined by the conditions $d\alpha(R_\alpha,\cdot)\equiv 0$
and $\alpha(R_\alpha)\equiv 1$ is called the {\it Reeb vector field} of $\alpha$.
The $3$-manifold $M$ equipped with a contact structure $\xi$ is called
a {\it contact manifold} and denoted by $(M,\xi)$.
Two contact manifolds $(M_1,\xi_1)$ and $(M_2,\xi_2)$ are said to be
{\it contactomorphic} if there exists a diffeomorphism $\varphi:M_1\to M_2$
such that $d\varphi:TM_1\to TM_2$ satisfies $d\varphi(\xi_1)=\xi_2$.
A disk $D$ in $(M,\xi)$ is called {\it overtwisted} if 
$D$ is tangent to $\xi$ at each point on $\bd D$.
If $(M,\xi)$ has an overtwisted disk then we say that $\xi$ is {\it overtwisted}
and otherwise that $\xi$ is {\it tight}. 

A graph multilink $L(\Gamma)$ is called {\it fibered}
if there is a fibration $\Sigma \setminus L(\Gamma)\to S^1$ such that 
\begin{itemize}
\item the intersection of the fiber surface and a small tubular neighborhood $N(S_i)$ of each 
link component $S_i$ of $L(\m)$
consists of $|m_i|>0$ leaves meeting along $S_i$ if $m_i\ne 0$, and
\item 
the working orientation of $S_i$ is consistent with (resp. opposite to) 
the orientation as the boundary of the fiber surface if $m_i>0$ (resp. $m_i<0$).
\end{itemize}

A fibered graph multilink $L(\Gamma)$ in $\Sigma$ is said to be {\it compatible} with 
a contact structure $\xi$ on $\Sigma$
if there exists a contact form $\alpha$ on $\Sigma$ whose kernel
is contactomorphic to $\xi$ and which satisfies that
$L(\Gamma)$ is positively transverse to $\ker\alpha$ and
$d\alpha$ is a volume form on the interiors of the fiber surfaces of $L(\Gamma)$;
in other words,
the Reeb vector field $R_\alpha$ of $\alpha$ is tangent to $L(\m)$
in the same direction and positively transverse to the interiors of the fiber surfaces
of $L(\m)$, see~\cite[Lemma~3.2]{ishikawa2}.

We remark that multilinks, fibered multilinks, and their compatible contact structures
are defined for any closed, oriented, smooth $3$-manifolds,
though we need to assign some working orientations to the link components at the beginning,
see~\cite{ishikawa2}. The same notion appears in~\cite{behm}, in which 
the fibration of a fibered multilink is called 
a {\it rational open book decomposition} of that $3$-manifold.
It is known that any fibered multilink in a closed, oriented, smooth $3$-manifold
admits a compatible contact structure
and two contact structures compatible with the same fibered multilink are contactomorphic,
see~\cite[Proposition~3.3 and~3.4]{ishikawa2} or~\cite[Theorem~1.7]{behm}.

We close this section with introducing a way of describing a contact structure
on $D^2\times S^1$, which we used in~\cite{ishikawa2}.
Let $\gamma$ be a curve on an $xy$-plane with parameter $r\in[0,1]$ 
which moves around $(0,0)$ in clockwise orientation.
For each point $\gamma(r)=(x(r),y(r))$ we set $(-h_1(r), h_2(r))=(x(r), y(r))$ and
define a $1$-form on $D^2\times S^1$ as $\alpha=h_2(r)d\mu+h_1(r)d\lambda$,
where $(r,\mu,\lambda)$ are coordinates of $D^2\times S^1$ with
polar coordinates $(r,\mu)$ of $D^2$, and $h_1$ and $h_2$ are real-valued smooth
functions with parameter $r$. 
Since the curve $\gamma$ rotates in clockwise orientation
it satisfies the inequality $h_1h_2'-h_2h_1'>0$, and this
implies the inequality $\alpha\land d\alpha>0$ except for the points at $r=0$.
Near $r=0$, we may assume that either $(-h_1,h_2)=(-c,r^2)$ or $(-h_1,h_2)=(c,-r^2)$
for some positive real number $c$, so that
$\alpha$ becomes a positive contact form on the whole $D^2\times S^1$.
See the right figure in Figure~\ref{fig0}. 

\begin{figure}[htbp]
   \centerline{\input{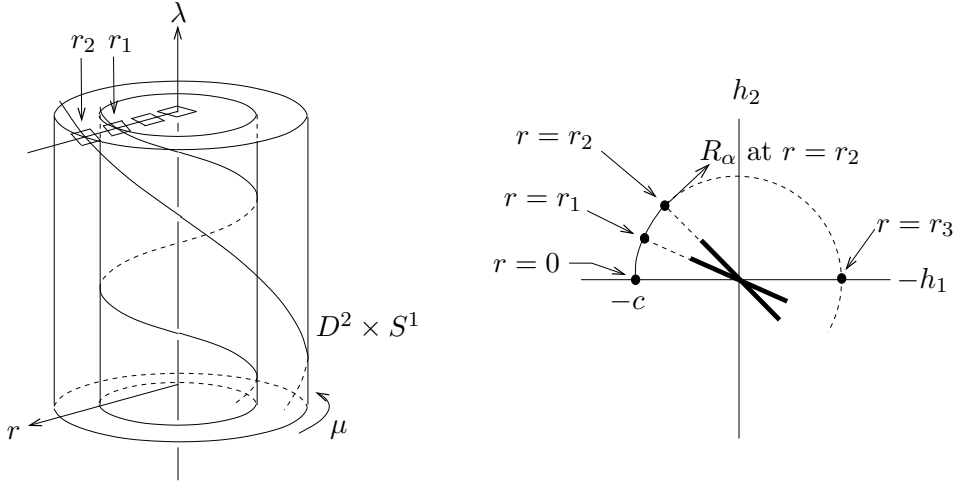}}
   \caption{How to read $\ker\alpha$ and $R_\alpha$ from the curve $\gamma(r)=(-h_1(r),h_2(r))$.\label{fig0}}
\end{figure}

Let $T(r)$ be a torus parallel to $\bd D^2\times S^1$ of radius $r>0$.
Since the positive normal vector to the contact structure $\ker\alpha$ is
$(h_2(r),h_1(r))$ on $T(r)$, the line on the $xy$-plane
connecting $(0,0)$ and $(-h_1(r),h_2(r))$ represents the slope of $\ker\alpha$
at $(r,\mu,\lambda)$.
Moreover, since the Reeb vector field of $\alpha$ is given as
\[
R_\alpha=\frac{1}{h_1h_2'-h_2h_1'}\left(-h_1'\frac{\bd}{\bd \mu}+h_2'\frac{\bd}{\bd \lambda}\right),
\]
the speed vector $\gamma'(r)=(-h_1'(r),h_2'(r))$ is parallel to $R_\alpha$ on $T(r)$
in the same direction.
If the curve $\gamma$ reaches the positive $x$-axis on the $xy$-plane, say at $r=r_3$,
then the contact structure $\ker\alpha$ has an overtwisted disk
$\{(r,\mu,\lambda)\in D^2\times S^1\mid \lambda=\text{constant},\;\,r\leq r_3\}$.
This is a typical example of overtwisted contact structures, called a {\it half Lutz twist}.

\section{Compatible contact structures in a special case}

Let $\Gamma$ be a minimal splice diagram not of type $\leftrightarrow$.
If a graph multilink $(\Sigma,L(\Gamma))$ is fibered then
the interiors of the fiber surfaces of $L(\Gamma)$ intersect
the fibers of the Seifert fibration in each Seifert piece transversely, see~\cite[Theorem~11.2]{en}
and the proof therein.
Note that the orientations of the fibers of the Seifert fibrations are
determine by formula~\eqref{eqab}.

\begin{dfn}
For a minimal splice diagram $\Gamma$ not of type $\leftrightarrow$,
we define $\hat\Gamma$ to be the diagram obtained from $\Gamma$
by applying the following modifications:
\begin{itemize}
\item[(1)] Replace each inner vertex of $\Gamma$ by $\oplus$ (resp. $\ominus$) if the fibers of 
the Seifert fibration in the corresponding Seifert piece is positively (resp. negatively) 
transverse to the fiber surface of $L(\Gamma)$.
\item[(2)] 
For each inner vertex $v$, assign $+$ (resp. $-$) 
to the root of each edge connected to $v$ if the multiplicity of the corresponding link
component of the Seifert multilink is positive or zero (resp. negative).
In particular, we assign $+$ to each edge with boundary vertex at the other endpoint
since it is regarded as an empty link component having multiplicity $0$.
\end{itemize}
\end{dfn}

Recall that we assumed that the denominators of the Seifert invariants 
of the Seifert fibered homology $3$-spheres before the splicing are all positive.

\begin{prop}\label{prop2000}
Let $(\Sigma,L(\Gamma))$ be a fibered PT graph multilink of a minimal splice diagram $\Gamma$
not of type $\leftrightarrow$.
Suppose that $\hat\Gamma$ has only $\oplus$ vertices and that all inner edges have only $+$ signs. 
If $L(\Gamma)$ has a component with negative multiplicity then its compatible contact structure
is overtwisted.
\end{prop}

We prove this assertion by showing the overtwisted contact structure 
explicitly according to the construction of a graph multilink in Section~2.
Let $C_i$ denote the boundary component $(-\bd \mathcal S)\cap D^2_i$ of $\mathcal S$.
The next lemma is a refinement of~\cite[Lemma~4.4]{ishikawa2}, 
where we added the conditions~(4) and~(5) for our purpose.

\begin{lemma}\label{lemma2012}
Suppose $A>0$ and $n\geq 2$. For $i=1,\cdots,n$, let $U_i$ be a collar neighborhood of $C_i$ in $\mathcal S$ 
with coordinates $(r_i,\theta_i)\in [1,2)\times S^1$
satisfying $\{(r_i,\theta_i)\mid r_i=1\}=C_i$.
Then, for a sufficiently small $\ve'>0$,
there exists a $1$-form $\beta$ on $\mathcal S$ which satisfies the following properties:
\begin{itemize}
\item[(1)] $d\beta>0$ on $\mathcal S$.
\item[(2)] If $i\geq 2$ and $\frac{b_i}{a_i}\leq 0$  then 
$\beta=R_ir_id\theta_i$ with $-\frac{b_i}{a_i}<R_i$ near $C_i$ on $U_i$.
\item[(3)] If $i\geq 2$ and $\frac{b_i}{a_i}>0$ then 
$\beta=\frac{R_i}{r_i}d\theta_i$ with $-\frac{b_i}{a_i}<R_i<0$ near $C_i$ on $U_i$.
\item[(4)] If $\frac{b_1}{a_1}\leq 0$ then 
$\beta=R_1r_1d\theta_1$ with $R_1=-\frac{b_1}{a_1}+\frac{1}{A}-\ve'>0$ near $C_1$ on $U_1$.
\item[(5)] If $\frac{b_1}{a_1}>0$ then 
$\beta=\frac{R_1}{r_1}d\theta_1$ with $R_1=-\frac{b_1}{a_1}+\frac{1}{A}-\ve'<0$ near $C_1$ on $U_1$.
\end{itemize}
\end{lemma}

\begin{proof}
Since $\left(-\frac{b_1}{a_1}+\frac{1}{A}\right)+\sum_{i=2}^k\left(-\frac{b_i}{a_i}\right)=0$,
we can choose $R_1,\cdots,R_k$ such that they satisfy the above conditions.
The rest of the proof is same as that of~\cite[Lemma~4.4]{ishikawa2}.
\end{proof}

\begin{lemma}\label{lemma2011}
Let $(\Sigma,L(\m))=(\Sigma(a_1,\cdots,a_k), m_1S_1\cup\cdots\cup m_nS_n)$ 
be a fibered PT Seifert multilink in a homology $3$-sphere $\Sigma$.
Suppose that the fibers of the Seifert fibration intersect the interiors of the
fiber surfaces of $L(\m)$ positively transversely.
For a sufficiently small positive real number $\ve$ given, 
there exists a positive contact form $\alpha$ on $\Sigma$ with the following properties:
\begin{itemize}
\item[(1)] $L(\m)$ is compatible with the contact structure $\ker\alpha$.
\item[(2)] The Reeb vector field $R_{\alpha}$ of $\alpha$ 
is tangent to the fibers of the Seifert fibration on $\mathcal S\times S^1$.
\item[(3)] On a neighborhood of $\bd (D^2\times S^1)_1$, 
$\alpha$ is given as $\alpha=h_{1,2}(r_1)d\mu_1+h_{1,1}(r_1)d\lambda_1$ with
$h_{1,1}(1)/h_{1,2}(1)=\ve$ and $h_{1,2}(1)>0$.
\item[(4)] On a neighborhood of $\bd (D^2\times S^1)_i$ for $i=2,\cdots,n$,
$\alpha=h_{i,2}(r_i)d\mu_i+h_{i,1}(r_i)d\lambda_i$ with $h_{i,2}(1)>0$.
\end{itemize}
Here $(r_i,\mu_i,\lambda_i)$ are coordinates of $(D^2\times S^1)_i$ chosen such that
$(r_i,\mu_i)$ are the polar coordinates of $D^2$ of radius $1$ and the orientation of $\lambda_i$ 
agrees with the working orientation of $S_i$,
and $h_{i,1}$ and $h_{i,2}$ are real-valued smooth functions with parameter $r_i$.
\end{lemma}

\begin{proof}
Let $B_i=[1,2)\times S^1\times S^1$ be a neighborhood of $\bd (D^2\times S^1)_i$ in
$\mathcal S\times S^1$ with coordinates $(r_i,\theta_i,t)$.
The solid torus $(D^2\times S^1)_i$ is glued to $B_i$ as
\[
   \mu_i \fM_i+\lambda_i \fL_i=(a_i\mu_i-\sigma_i\lambda_i)Q_i+(b_i\mu_i+\delta_i\lambda_i)H,
\]
where $(\fM_i,\fL_i)$ is the standard meridian-longitude pair on $\bd (D^2\times S^1)_i$,
$Q_i$ is the oriented curve given by $\{1\}\times S^1\times\{\text{a point}\}\subset\bd B_i$,
$H$ is a typical fiber of the projection $[1,2)\times S^1\times S^1\to [2,1)\times S^1$
which omits the third entry, and $a_i, b_i, \sigma_i, \delta_i\in\Z$ are
given according to relations~\eqref{eqab}.
As in the proof of~\cite[Proposition~4.1]{ishikawa2},
the contact form on $\mathcal S\times S^1$ is set as $\alpha_0=\beta+dt$, 
where $\beta$ is a $1$-form chosen in Lemma~\ref{lemma2012}.
The Reeb vector field $R_{\alpha_0}=\frac{\bd}{\bd t}$ satisfies the condition~(1)
in the assertion on $\mathcal S\times S^1$.

For the gluing map $\varphi_i$ of $(D^2\times S^1)_i$ to $B_i$, 
$\varphi_i^*\alpha_0$ is given as
\[
\begin{split}
   \varphi^*_i\alpha_0
   &=R_ir_id(a_i\mu_i-\sigma_i\lambda_i)+d(b_i\mu_i+\delta_i\lambda_i)
    =(b_i+a_iR_ir_i)d\mu_i+(\delta_i-\sigma_i R_ir_i)d\lambda_i \\
   &=a_i\left(\frac{b_i}{a_i}+R_i r_i\right)d\mu_i
   +\frac{1}{a_i}\left(1-a_i\sigma_i\left(\frac{b_i}{a_i}+R_ir_i\right)\right)d\lambda_i.
\end{split}
\]
Hence the inequality $h_{i,2}(1)>0$ holds for $i=1,\cdots,k$.
Moreover, since $R_1$ is chosen as
$R_1=-\frac{b_1}{a_1}+\frac{1}{A}-\ve'$ in Lemma~\ref{lemma2012}, we have
\[
\begin{split}
    h_{1,1}(1)&=\frac{1}{a_1}\left(1-a_1\sigma_1\left(\frac{b_1}{a_1}+R_1\right)\right) 
          =\frac{1}{a_1}\left(1-A\left(\frac{1}{A}-\ve'\right)\right) 
          =\frac{A\ve'}{a_1}, \\
   h_{1,2}(1)&=a_1\left(\frac{b_1}{a_1}+R_1\right) 
         =a_1\left(\frac{1}{A}-\ve'\right)
\end{split}
\]
and hence
\[
   \frac{h_{1,1}(1)}{h_{1,2}(1)}=\frac{A\ve'}{a_1^2\left(\frac{1}{A}-\ve'\right)}>0.
\]
Since $\lim_{\ve'\to 0}h_{1,1}(1)/h_{1,2}(1)=0$, we can choose $\ve'>0$ such that $h_{1,1}(1)/h_{1,2}(1)=\ve$.

We finally extend the contact form $\alpha_0$ on $\mathcal S\times S^1$ into each $(D^2\times S^1)_i$.
If $m_i>0$ then we describe a curve $\gamma_i(r_i)=(-h_{i,1}(r_i), h_{i,2}(r_i))$ on 
the $xy$-plane representing a positive contact form on $(D^2\times S^1)_i$ in such a way that
\begin{itemize}
\item $(-h_{1,i}, h_{2,i})=(-c_i, r_i^2)$ near $r_i=0$ with some constant $c_i>0$,
\item $h_{i,2}d\mu_i+h_{i,1}d\lambda_i=\varphi_i^*\alpha_0$ near $r_i=1$, and
\item $\gamma_i'(r_i)$ rotates monotonously.
\end{itemize}
The Reeb vector field of this contact form is positively transverse to the interiors of
the fiber surfaces of $L(\m)$ on $(D^2\times S^1)_i$ as shown in Figure~\ref{fig1d}.
The same observation works even in the case where $m_i=0$.
\begin{figure}[htbp]
   \centerline{\input{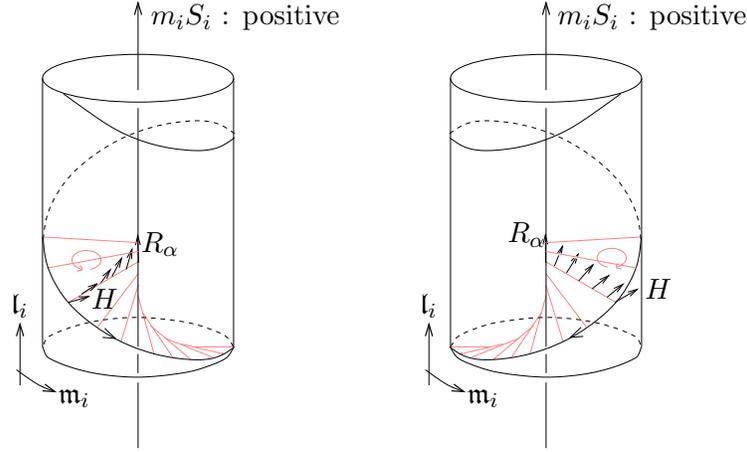}}
   \caption{The compatibility in $(D^2\times S^1)_i$ with $m_i>0$.\label{fig1d}}
\end{figure}

If $m_i<0$ then we describe a curve on
the $xy$-plane representing a positive contact form on $(D^2\times S^1)_i$ in such a way that
\begin{itemize}
\item $(-h_{1,i}, h_{2,i})=(c_i, -r_i^2)$ near $r_i=0$ with some constant $c_i>0$,
\item $h_{i,2}d\mu_i+h_{i,1}d\lambda_i=\varphi_i^*\alpha_0$ near $r_i=1$, and
\item $\gamma_i'(r_i)$ rotates monotonously.
\end{itemize}
In this case, we also have the positive transversality as shown in Figure~\ref{fig1e}.
This completes the proof.
\end{proof}

\begin{figure}[htbp]
   \centerline{\input{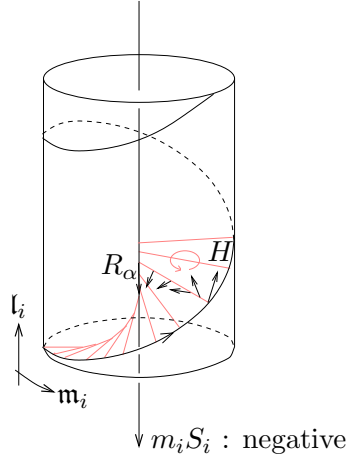}}
   \caption{The compatibility in $(D^2\times S^1)_i$ with $m_i<0$.\label{fig1e}}
\end{figure}

\begin{lemma}\label{lemma2010}
Let $(\Sigma,L(\Gamma))$ be a fibered PT graph multilink of a minimal splice diagram $\Gamma$
not of type $\leftrightarrow$.
Suppose that $\hat\Gamma$ has only $\oplus$ vertices and that all inner edges have only $+$ signs. 
Then there exists a positive contact form $\alpha$ on $\Sigma$ with the following properties:
\begin{itemize}
\item[(1)] $L(\Gamma)$ is compatible with the contact structure $\ker\alpha$.
\item[(2)] Each component $m_iS_i$ of $L(\Gamma)$ with negative multiplicity
has a tubular neighborhood which contains a half Lutz twist.
In particular, it contains an overtwisted disk.
\end{itemize}
\end{lemma}

\begin{proof}
Let $L_1(\m_1),\cdots,L_{\ell'}(\m_{\ell'})$ denote the fibered Seifert multilinks before the 
splicing of $L(\Gamma)$.
Suppose that $L(\Gamma)$ is obtained from $L_1(\m_1)$ by splicing $L_2(\m),\cdots,L_{\ell'}(\m_{\ell'})$
successively. The assertion for $L_1(\m_1)$ had been proved in~\cite[Proposition~4.1]{ishikawa2}.
Assume that the assertion holds for the fibered PT graph multilink $(\Sigma_{\ell-1},L(\Gamma_{\ell-1}))$
obtained from $L_1(\m_1)$ by splicing $L_2(\m_2),\cdots,L_{\ell-1}(\m_{\ell-1})$ successively,
where $2\leq\ell\leq\ell'$, and then consider the next splice 
\[
     (\Sigma_\ell,L(\Gamma_\ell))=\left[(\Sigma_{\ell-1},L(\Gamma_{\ell-1})) \textstyle\frac{\;\;\;\;\;\;}{S\;\;\;\;\;S_{\ell,1}} (\Sigma', L_\ell(\m_\ell))\right],
\]
where $\Sigma'$ is the Seifert fibered homology $3$-sphere containing $L_\ell(\m_\ell)$.

Let $\hat\alpha_\ell$ be a contact form on $\Sigma'$ obtained according to Lemma~\ref{lemma2011}, where $(D^2\times S^1)_1$ in the lemma corresponds to the neighborhood $N(S_{\ell,1})$ of
$S_{\ell,1}$ for the splicing. From the lemma, $\hat\alpha_\ell$ has the form 
$\hat\alpha_\ell=h_{\ell,2}(r_\ell)d\mu_\ell+h_{\ell,1}(r_\ell)d\lambda_\ell$
near $\bd N(S_{\ell,1})$
with $h_{\ell,1}(1)/h_{\ell,2}(1)>0$ being sufficiently small and $h_{\ell,2}(1)>0$,
where $(r_\ell,\mu_\ell,\lambda_\ell)$ are coordinates of $N(S_{\ell,1})=D^2\times S^1$ chosen such that
$(r_\ell,\mu_\ell)$ are the polar coordinates of $D^2$ of radius $1$ and
the orientation of $\lambda_\ell$ agrees with the working orientation of $S_{\ell,1}$.

On the other hand, there is a contact form $\alpha_{\ell-1}$ on $\Sigma_{\ell-1}$ 
satisfying the required properties by the assumption of the induction.
Set the neighborhood $N(S)$ of $S$ for the splicing to be the neighborhood
specified in Lemma~\ref{lemma2011} in the inductive construction of $L(\Gamma_{\ell-1})$.
Then, on a small neighborhood of $\bd N(S)$, $\alpha_{\ell-1}$ has the form 
\begin{equation}\label{eq1001}
   \alpha_{\ell-1}=h_2(r)d\mu+h_1(r)d\lambda
\end{equation}
with $h_2(1)>0$,
where $(r,\mu,\lambda)$ are coordinates of $N(S)=D^2\times S^1$ chosen such that
$(r,\mu)$ are the polar coordinates of $D^2$ of radius $1$ and
the orientation of $\lambda$ agrees with the working orientation of $S$.
Set the gluing map of the splice as
$(r_\ell,\mu_\ell,\lambda_\ell)=(2-r,\lambda,\mu)$, then
on a small neighborhood $N(T)$ of the torus $T=\bd N(S_{\ell,1})=\bd N(S)$ for the splicing
 we have
\begin{equation}\label{eq1002}
   \hat\alpha_\ell=h_{\ell,1}(2-r)d\mu+h_{\ell,2}(2-r)d\lambda.
\end{equation}

Now we plot the two points corresponding to the contact forms~\eqref{eq1001} and~\eqref{eq1002}
on the $xy$-plane and describe a curve connecting them to obtain a contact form on $N(T)$
gluing $\alpha_{\ell-1}$ and $\hat\alpha_\ell$ smoothly.
The Reeb vector fields of $\alpha_{\ell-1}$ and $\hat\alpha_\ell$ were chosen such that
they are positively transverse to the fiber surfaces of $L(\Gamma_\ell)$ in $N(T)$.
Recall that
$h_2(1)>0, h_{\ell,2}(1)>0$ and that $h_{\ell,1}(1)/h_{\ell,2}(1)>0$ can be sufficiently small.
Since $\hat\alpha_{\ell-1}$ at $r=1$ corresponds to the point
$(-h_{\ell,2}(1), h_{\ell,1}(1))$ on the $xy$-plane representing a contact form on $N(S)$,
by choosing $h_{\ell,1}(1)/h_{\ell,2}(1)>0$ sufficiently small and
multiplying a positive constant to $\hat\alpha_\ell$ if necessary,
we can describe a curve $\gamma(r)$ on that $xy$-plane which defines a positive 
contact form $\alpha_{N(T)}$ on $N(T)$ connecting $\alpha_{\ell-1}$ and $\hat\alpha_\ell$ 
smoothly and whose speed vector $\gamma'(r)$ rotates monotonously, see Figure~\ref{fig1b}. 
Set the slope of the fiber surface of $L(\Gamma_\ell)$ to be constant in $N(T)$, then
the monotonous rotation of $\gamma'(r)$ ensures that $\ker\alpha_{N(T)}$ is
compatible with $L(\Gamma_\ell)$ on $N(T)$. 
We then glue $\alpha_{\ell-1}$ and $\hat\alpha_\ell$ by $\alpha_{N(T)}$
and obtain a contact form $\alpha_\ell$ on $\Sigma_\ell$.
Since the Reeb vector field of $\alpha_\ell$ satisfies the compatibility condition
outside $N(T)$, $\ker\alpha_\ell$ is compatible with $L(\Gamma_\ell)$ on
the whole $\Sigma_\ell$.
\begin{figure}[htbp]
   \centerline{\input{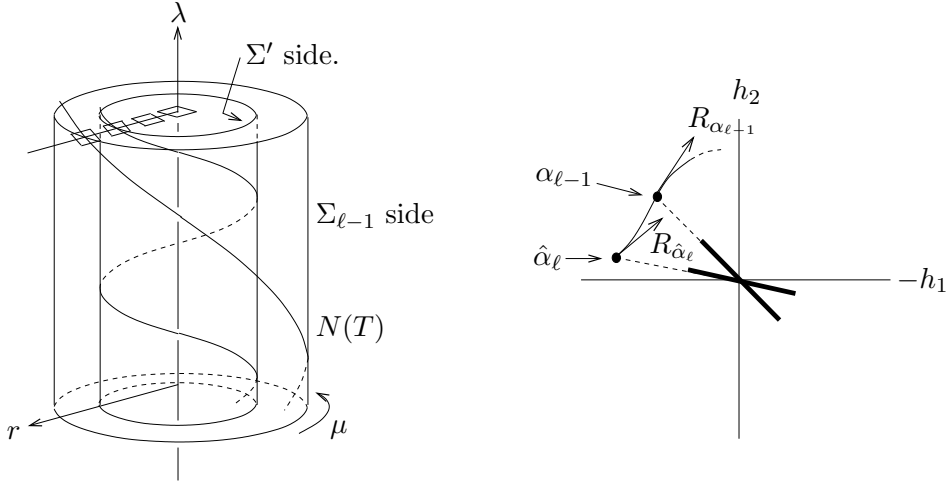}}
   \caption{Connect the contact forms $\hat\alpha_{\ell-1}$ and $\alpha_{\ell}$ smoothly.\label{fig1b}}
\end{figure}

By induction, we obtain a contact form $\alpha$ whose kernel is compatible with $L(\Gamma)$
and which satisfies the inequality $h_2(1)>0$ for each link component of $L(\Gamma)$.
If the link component is negative then the starting point of the curve on the $xy$-plane
becomes the point $(c,0)$ with a positive constant $c$. Therefore,
the inequality $h_2(1)>0$ ensures that $\ker\alpha$ has a half Lutz twist
in the neighborhood of that negative component.
\end{proof}

\vspace{3mm}

\noindent
{\it Proof of Proposition~\ref{prop2000}.\;\,}
The assertion follows from Lemma~\ref{lemma2010}.
\qed

\section{Compatible contact structures via the Thurston-Winkelnkemper's construction}

In the proof of Theorem~\ref{thm01}, we will compare two compatible contact structures;
one is constructed in the previous section and 
the other is obtained by applying the Thurston-Winkelnkemper's construction
to the fibration of the fibered multilink such that the contact form is ``standard''
on the tori for the splicings, which we will show in this section.

We prepare one lemma before showing the construction.

\begin{lemma}\label{lemmatw}
Let $L(\m)$ be a fibered PT Seifert multilink in a homology $3$-sphere $\Sigma$
with link components $m_1S_1,\cdots,m_nS_n$ with $n\geq 2$ and $m_i\ne 0$,
and let $c_1,\cdots,c_n$ be positive real numbers satisfying
the inequality $-c_1+\sum_{i=2}^n c_i>0$.
Then there exists a contact form $\alpha$ on $\Sigma$ satisfying the following properties:
\begin{itemize}
\item[(1)] $L(\m)$ is compatible with the contact structure $\ker\alpha$.
\item[(2)] On a neighborhood of $\bd (D^2\times S^1)_1$, $\alpha$ has the form
\[
   \alpha=RU_1d\mu_1+\left(-c_1r_1+RV_1\right)d\lambda_1
\]
with $U_1>0$.
\item[(3)] On a neighborhood of $\bd (D^2\times S^1)_i$ for $i=2,\cdots,n$, $\alpha$ has the form
\[
   \alpha=RU_id\mu_i+\left(\frac{c_i}{r_i}+RV_i\right)d\lambda_i
\]
with $U_i>0$.
\end{itemize}
Here $R$ is a sufficiently large real number, $(r_i,\mu_i,\lambda_i)$ are coordinates
of $(D^2\times S^1)_i$ chosen such that $(r_i,\mu_i)$ are the polar coordinates 
of $D^2$ of radius $1$ and the orientation of $\lambda_i$ 
agrees with that of $m_iS_i$, and $(U_i,V_i)$ is a vector positively normal
to the fiber surface on $\bd (D^2\times S^1)_i$ with coordinates $(\mu_i,\lambda_i)$.
\end{lemma}

\begin{proof}
We denote by $F_t$ the fiber surface of $L(\m)$ over $t\in S^1=[0,1]/0\sim 1$ and
choose a diffeomorphism $\phi_t:F_0\to F_t$ of the fibration of $L(\m)$
in such a way that
\[
\phi_t(r_i,\mu_i,\lambda_i)=\left(r_i, \mu_i+\frac{t}{|m_i|} , \lambda_i \right)
\]
on each $(D^2\times S^1)_i$.
Let $\theta_i$ be the coordinate function on the curve $(-F_0)\cap \bd (D^2\times S^1)_i$
given as $\theta_i=-\lambda_i$. Set $\hat F_0=F_0\cap (\mathcal S\times S^1)$ and
let $\Omega$ be a volume form on $\hat F_0$ which satisfies
\begin{itemize}
\item $\int_{\hat F_0}\Omega=-c_1+\sum_{i=2}^n c_i>0$,
\item $\Omega=c_1dr_1\land d\theta_1$ near $\hat F_0\cap \bd (D^2\times S^1)_1$, and
\item $\Omega=\frac{c_i}{r_i^2}dr_i\land d\theta_i$ near $\hat F_0\cap \bd (D^2\times S^1)_i$
for $i=2,\cdots,n$.
\end{itemize}
Then, as in~\cite{tw}, we can find a $1$-form $\beta$ on $\hat F_0$ such that 
\begin{itemize}
\item $d\beta$ is a volume form on $\hat F_0$,
\item $\beta=c_1r_1d\theta_1=-c_1r_1d\lambda_1$ near $\hat F_0\cap \bd (D^2\times S^1)_1$, and
\item $\beta=-\frac{c_i}{r_i}d\theta_i=\frac{c_i}{r_i}d\lambda_i$ near $\hat F_0\cap \bd (D^2\times S^1)_i$ for $i=2,\cdots,n$.
\end{itemize}
The manifold $M$ is constructed from 
$\hat F_0\times [0,1]$ by identifying $(x,1)\sim (\phi_1(x),0)$ for each $x\in \hat F_0$
and then filling the boundary components by the solid tori $(D^2\times S^1)_i$'s.
Using this construction, we define a $1$-form $\alpha_0$ on $\mathcal S\times S^1$ as
\[
   \alpha_0=(1-t)\beta+t\phi_1^*(\beta)+Rdt,
\]
where $R>0$.
Near the boundary component $\bd (D^2\times S^1)_i$, this $1$-form is given as
\[
   \alpha_0=\beta+Rdt=
\begin{cases}
-c_1r_1d\lambda_1+R(U_1d\mu_1+V_1d\lambda_1) & \\
\frac{c_i}{r_i}d\lambda_i+R(U_id\mu_i+V_id\lambda_i) & \text{for\;\,} i=2,\cdots,n. \\
\end{cases}
\]
We choose $R$ sufficiently large such that $\alpha_0$ becomes 
a positive contact form on $\mathcal S\times S^1$.
Since $\lambda_i$ is oriented in the same direction as the oriented link component
$m_iS_i$, we always have $U_i>0$ for $i=1,\cdots,n$, see Figure~\ref{fig11}.
Hence the conditions~(2) and (3) are satisfied. 
\begin{figure}[htbp]
   \centerline{\input{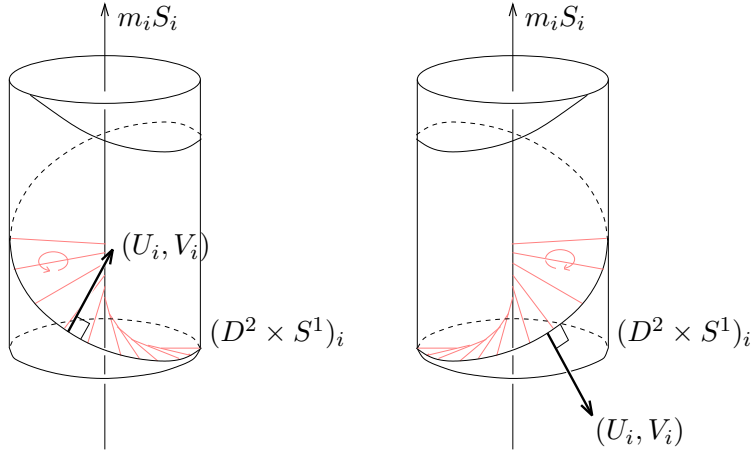}}
   \caption{The positive normal vector of the fiber surface along $\bd (D^2\times S^1)_i$.\label{fig11}}
\end{figure}

Finally we extend the contact form $\alpha_0$ into each $(D^2\times S^1)_i$.
The contact form $\alpha_0$ at $r_i=1$ corresponds to the point
$(-RV_1+c_1, RU_1)$ and $(-RV_i-c_i, RU_i)$ for $i=2,\cdots,n$
and its Reeb vector field corresponds to the vector $(c_i,0)$, for $i=1,\cdots,n$,
on the $xy$-plane used to represent the contact form on $(D^2\times S^1)_i$.
Since $U_i>0$, $c_i>0$ and $R$ is sufficiently large, we can describe 
a curve on the $xy$-plane as we did in the end of the proof of Lemma~\ref{lemma2011},
which gives a contact form $\alpha$ on $\Sigma$. On $\mathcal S\times S^1$,
since the Reeb vector field $R_\alpha$ of $\alpha$ is given as $\frac{1}{R}\frac{\bd}{\bd t}$,
it is positively transverse to the fiber surfaces of $L(\m)$.
This property is also satisfied on each $(D^2\times S^1)_i$ since
$R_\alpha$ and the fiber surface of $L(\m)$ are set as shown in Figure~\ref{fig1d}.
Thus $\ker\alpha$ is compatible with $L(\m)$ on the whole $\Sigma$.
\end{proof}

\begin{prop}\label{lemma031}
Let $(\Sigma, L(\Gamma))$ be a fibered PT graph multilink in a homology $3$-sphere $\Sigma$.
There exists a contact form $\alpha$ with the following properties:
\begin{itemize}
\item[(1)] $L(\Gamma)$ is compatible with the contact structure $\ker\alpha$.
\item[(2)] Let
\[
     (\Sigma,L(\Gamma))=\left[(\Sigma_1,L(\Gamma_1)) \textstyle\frac{\;\;\;\;\;\;}{S_1\;\;\;\;\;S_2} (\Sigma_2, L(\Gamma_2))\right]
\]
be a splice of $(\Sigma, L(\Gamma))$.
On the neighborhood $N(S_1)$ of $S_1$ for the splicing,
$\alpha$ has the form $\alpha=h_2(r)d\mu+h_1(r)d\lambda$, where
$(r,\mu,\lambda)$ are coordinates of $N(S_1)=D^2\times S^1$ chosen such that
$(r,\mu)$ are the polar coordinates of $D^2$ of radius $1$ and
the orientation of $\lambda$ agrees with that of $m_1S_1$
if $m_1\ne 0$ and is positively transverse to the meridional disk if $m_1=0$, and
$h_1$ and $h_2$ are real-valued smooth functions with parameter $r\in [0,1]$ such that
the argument of $-h_1(r)+\sqrt{-1}h_2(r)$ varies in $(0,\pi]$.
\end{itemize}
\end{prop}

\begin{proof}
Let $L_1(\m_1),\cdots,L_{\ell'}(\m_{\ell'})$ denote the fibered Seifert multilinks before the
splicing of $L(\Gamma)$.
Suppose that $L(\Gamma)$ is obtained from $L_1(\m_1)$ by splicing 
$L_2(\m_2),\cdots,L_{\ell'}(\m_{\ell'})$ successively and, for $i=2,\cdots,\ell'$,
let $S_{i,1}$ denote the link component of $L_i(\m_i)$ for this inductive splicings.

We prove the assertion by induction.
The assertion holds for $L_1(\m_1)$ as proved in Lemma~\ref{lemmatw}.
Suppose that the assertion holds for a fibered PT graph multilink $(\Sigma_{\ell-1},L(\Gamma_{\ell-1}))$ 
obtained from $L_1(\m_1)$ by splicing $L_2(\m_2),\cdots,L_{\ell-1}(\m_{\ell-1})$ successively, 
where $2\leq \ell\leq \ell'$,
and consider the next splice
\[
     (\Sigma_\ell,L(\Gamma_\ell))=\left[(\Sigma_{\ell-1},L(\Gamma_{\ell-1})) \textstyle\frac{\;\;\;\;\;\;}{S\;\;\;\;\;S_{\ell,1}} (\Sigma', L_\ell(\m_\ell))\right].
\]
Let $\alpha_{\ell-1}$ denote the contact form on $\Sigma_{\ell-1}$ satisfying the required
properties in the assertion, and
$m$ and $m_{\ell,1}$ denote the multiplicities of $S$ and $S_{\ell,1}$ respectively.

We first deal with the case where $m\ne 0$ and $m_{\ell,1}\ne 0$.
Let $\hat\alpha_{\ell}$ be the contact form on $\Sigma'$
whose kernel is compatible with $L_{\ell}(\m_{\ell})$,
obtained by applying Lemma~\ref{lemmatw} such that the specified component $S_1$ in the lemma
corresponds to the link component $S_{\ell,1}$ of $L_{\ell}(\m_{\ell})$.
Set the neighborhood $N(S_{\ell,1})$ of $S_{\ell,1}$ for the splicing to be 
$(D^2\times S^1)_1$ in the lemma and
fix coordinates $(r_{\ell,1}, \mu_{\ell,1}, \lambda_{\ell,1})$ of
$N(S_{\ell,1})=D^2\times S^1$ such that $(r_{\ell,1}, \mu_{\ell,1})$ are the polar
coordinates of $D^2$ of radius $1$ and the orientation of $\lambda_{\ell,1}$
agrees with that of $m_{\ell,1}S_{\ell,1}$.
On $N(S_{\ell,1})$, the contact form $\hat\alpha_{\ell}$ is given as
\[
   \hat\alpha_\ell
=R_\ell U_{\ell,1}d\mu_{\ell,1}+\left(-c_{\ell,1}r_{\ell,1}+R_\ell V_{\ell,1}\right)d\lambda_{\ell,1}
\]
with $U_{\ell,1}>0$,
where $(U_{\ell,1},V_{\ell,1})$ is a vector positively normal to the fiber surface of 
$L_\ell(\m_\ell)$ on $\bd N(S_{\ell,1})$ with coordinates $(\mu_{\ell,1},\lambda_{\ell,1})$
and $c_{\ell,1}$ is some positive constant.

Similarly, set the neighborhood $N(S)$ of $S$ for the splicing to be the one
specified in Lemma~\ref{lemmatw} in the inductive construction of $L(\Gamma_{\ell-1})$
and fix coordinates $(r,\mu,\lambda)$ of $N(S)=D^2\times S^1$ such that
$(r,\mu)$ are the polar coordinates of $D^2$ of radius $1$ and the orientation of $\lambda$
agrees with that of $mS$.
On a neighborhood of $\bd N(S)$, the contact form $\alpha_{\ell-1}$ is given as
\[
   \alpha_{\ell-1}=RUd\mu+\left(\frac{c}{r}+RV\right)d\lambda
\]
with $U>0$,
where $(U,V)$ is a vector positively normal to the fiber surface of $L(\Gamma_{\ell})$
on $\bd N(S)$ with coordinates $(\mu,\lambda)$ and $c$ is some positive constant.
The gluing map of the splice can be written as 
$(r, \mu, \lambda)=(2-r_{\ell,1},\pm\lambda_{\ell,1},\pm\mu_{\ell,1})$, where
the sign $\pm$ depends on if $\lambda_{\ell,1}$ and 
$\lambda$ are consistent with the working orientations or not, which will be clarified later.
Let $T$ denote the torus $\bd N(S_{\ell,1})=\bd N(S)$ for the splicing and
$N(T)$ denote its small neighborhood.

Consider the case $V_{\ell,1}>0$. 
We assume that the orientation of $\lambda$ agrees with the working orientation of $S_{\ell,1}$.
The argument below works in the case where they are opposite, so we omit the proof in that case.
In the case under consideration,
since these orientations coincide, we have the left figure in Figure~\ref{fig2}.
By observing the identification of the splicing, 
the orientation of the fiber surface is fixed as shown on the right,
which implies that the orientation of $\lambda$ agrees with the working orientation of $S$.
Hence the positive normal
vectors $(U_{\ell,1}, V_{\ell,1})$ and $(U,V)$ of the fiber surface
are identified on $T$ as $(U,V)=K(V_{\ell,1},U_{\ell,1})$ 
with some constant $K>0$, see Figure~\ref{fig2}.
This means that the sign $\pm$ above is $+$ and
the gluing map is given as $(r, \mu, \lambda)=(2-r_{\ell,1}, \lambda_{\ell,1}, \mu_{\ell,1})$.
Thus the contact form $\alpha_{\ell-1}$ is written as
\[
   \alpha_{\ell-1}=\left(\frac{c}{2-r_{\ell,1}}+RKU_{\ell,1}\right)d\mu_{\ell,1}
+RKV_{\ell,1}d\lambda_{\ell,1}.
\]

We now choose $R$ sufficiently large relative to $R_\ell$ such that
the positive normal vector of $\ker\alpha_{\ell-1}$ lies between
that of $\ker\hat\alpha_\ell$ and that of the fiber surface, see Figure~\ref{fig3}.
Note that the figure is
described with the coordinates $(r_{\ell,1}, \mu_{\ell,1}, \lambda_{\ell,1})$.
By multiplying a positive constant to $\hat\alpha_\ell$ if necessary,
we can describe a curve $\gamma(r)$ on the $xy$-plane which represents
a positive contact form $\alpha_{N(T)}$ on $N(T)$ connecting $\alpha_{\ell-1}$
and $\hat\alpha_\ell$ smoothly.  Moreover we can choose $\gamma(r)$ such that
$\gamma'(r)$ rotates monotonously.
Set the slope of the fiber surface of $L(\Gamma_\ell)$ to be constant in $N(T)$, then
the monotonous rotation of $\gamma'(r)$ ensures that $\alpha_{N(T)}$ is
compatible with $L(\Gamma_\ell)$. 
Figure~\ref{fig3} shows that the argument of $-h_1(r)+\sqrt{-1}h_2(r)$ varies in $(0,\pi]$.

\begin{figure}[htbp]
   \centerline{\input{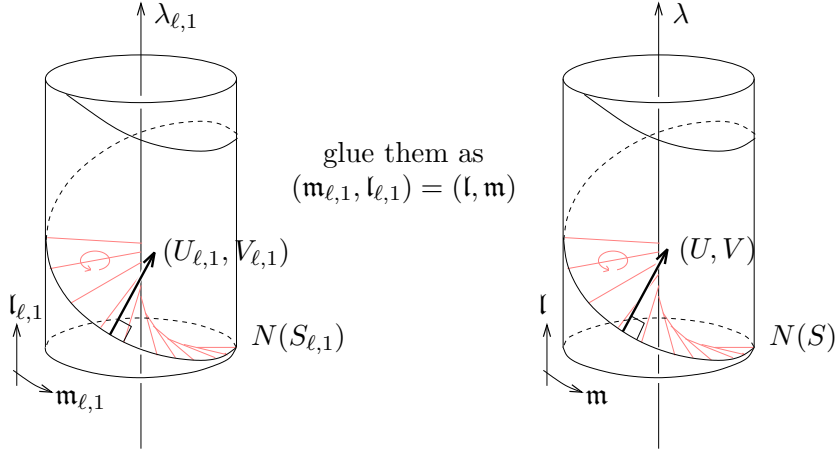}}
   \caption{The positive normal vector of the fiber surface on $T$ in case $V_{\ell,1}>0$.
Here $(\fM_{\ell,1},\fL_{\ell,1})$ and $(\fM, \fL)$ are the preferred meridian-longitude
pairs of $\Sigma'\setminus\text{int\,} N(S_{\ell,1})$ and
 $\Sigma_{\ell-1}\setminus\text{int\,}N(S)$ for the splicing respectively.\label{fig2}}
\end{figure}

\begin{figure}[htbp]
   \centerline{\input{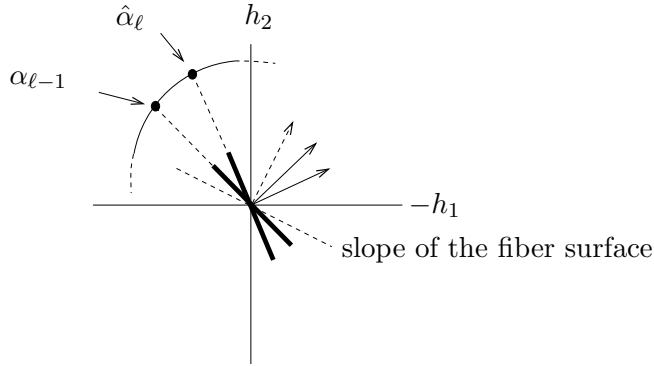}}
   \caption{Glue $\alpha_{\ell-1}$ and $\hat\alpha_\ell$ on $N(T)$ in case $V_{\ell,1}>0$.\label{fig3}}
\end{figure}

We next consider the case $V_{\ell,1}<0$.
Assume again that the orientation of $\lambda_{\ell,1}$ agrees with
the working orientation of $S_{\ell,1}$, and we omit the proof in the other case.
In this case,
the mutual positions of the fiber surface, $\lambda_{\ell,1}$ and $\lambda$ become
as shown in Figure~\ref{fig4}.
In particular, the orientation of $\lambda$ is opposite to the working orientation of $S$.
Let $\overset{\to}{n}$ be the vector positive normal to $\ker\alpha_{\ell-1}$, which is
given as $(U,V)$ with the coordinates $(\mu,\lambda)$.
This means that $\overset{\to}{n}$ is given as $(-U,-V)$ with
the coordinates corresponding to the preferred meridian-longitude pair $(\fM,\fL)$
of $\Sigma_{\ell-1}\setminus\text{int\,}N(S)$ for the splicing. Hence,
the positive normal vectors $(U_{\ell,1}, V_{\ell,1})$ and $(U,V)$ of the fiber surface
are identified on $T$ as $(U,V)=-K(V_{\ell,1},U_{\ell,1})$ with some constant $K>0$.
Since the gluing map of the splicing is given as
$(r,\mu,\lambda)=(2-r_{\ell,1},-\lambda_{\ell,1},-\mu_{\ell,1})$,
we have
\[
   \alpha_{\ell-1}=\left(-\frac{c}{2-r_{\ell,1}}+RKU_{\ell,1}\right)d\mu_{\ell,1}
+RKV_{\ell,1}d\lambda_{\ell,1}.
\]
Plot $\alpha_{\ell,1}$ above and $\hat\alpha_\ell$ on the $xy$-plane 
and obtain a contact form connecting $\alpha_{\ell-1}$ and $\hat\alpha_\ell$ on $N(T)$ smoothly
such that it satisfies the required conditions, see Figure~\ref{fig6}.
Note that the figure is described with
the coordinates $(r_{\ell,1}, \mu_{\ell,1}, \lambda_{\ell,1})$.
The figure shows that the argument of $-h_1(r)+\sqrt{-1}h_2(r)$ varies in $(0,\pi]$.
This completes the proof in case $V_{\ell,1}<0$.
\begin{figure}[t]
   \centerline{\input{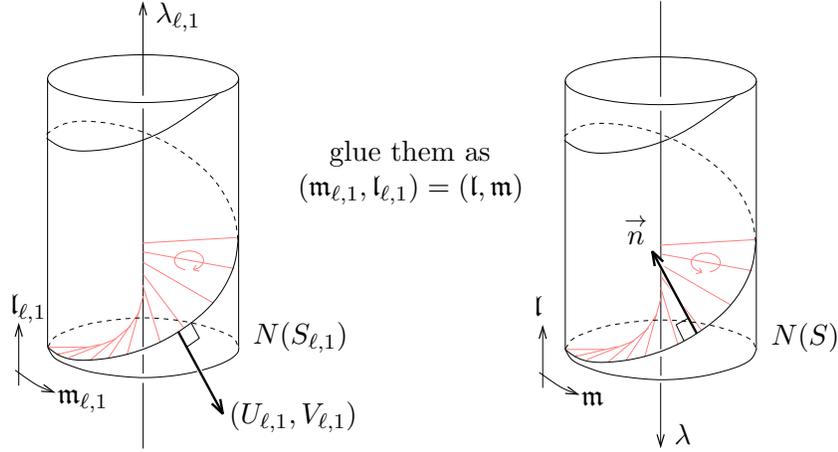}}
   \caption{The positive normal vector of the fiber surface on $T$ in case $V_{\ell,1}<0$.\label{fig4}}
\end{figure}

\begin{figure}[t]
   \centerline{\input{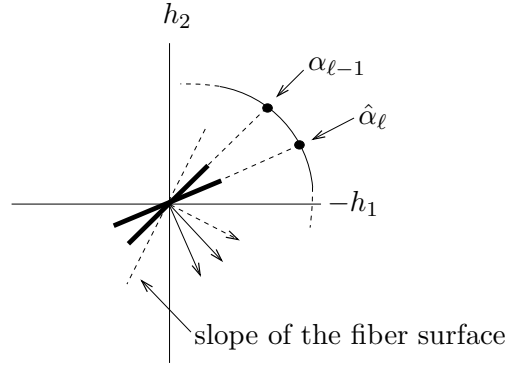}}
   \caption{Glue $\alpha_{\ell-1}$ and $\hat\alpha_\ell$ on $N(T)$ in case $V_{\ell,1}<0$.\label{fig6}}
\end{figure}

To finish the proof of the proposition, we need to prove
the assertion in case $m_{\ell,1}=0$ and case $m=0$.
We prove the former case here and omit the latter since the proof is similar.
We further assume that the orientation of $\lambda_{\ell,1}$ agrees with
the working orientation of $S_{\ell,1}$ and omit the proof of the other case.
In the case under consideration,
we can set $(U_{\ell,1}, V_{\ell,1})=(0,1)$ and $(U, V)=(1,0)$.
In particular, we have $m>0$.
Let $\hat\alpha_\ell$ be a contact form on $\Sigma'$ obtained by applying Lemma~\ref{lemmatw}.
The empty link component $S_{\ell,1}$ has a tubular neighborhood $N(S_{\ell,1})$
with a contact form
\[
   \hat\alpha_\ell=c_{\ell,1}(r^2_{\ell,1}d\mu_{\ell,1}+d\lambda_{\ell,1}),
\]
where $c_{\ell,1}$ is a positive constant and the radius of $N(S_{\ell,1})$, say $\ve>0$,
is sufficiently small.
We choose the neighborhood $N(S)$ of $S$ for the splicing as before.
Now we perform the splice for the above $N(S_{\ell,1})$ and $N(S)$.
The gluing map is given as $(r,\mu,\lambda)=(1+\ve-r_{\ell,1}, \lambda_{\ell,1}, \mu_{\ell,1})$
and hence the contact form $\alpha_{\ell-1}$ is
\[
   \alpha_{\ell-1}=\frac{c}{1+\ve-r_{\ell,1}}d\mu_{\ell,1}+Rd\lambda_{\ell,1}.
\]
We choose $R$ sufficiently large such that the positive normal vector of $\ker\alpha_{\ell-1}$
lies between that of $\ker\hat\alpha_\ell$ and that of the fiber surface, see Figure~\ref{fig5}.
Then the two contact forms $\alpha_{\ell-1}$ and $\hat\alpha_\ell$ are connected smoothly
by describing a curve on the $xy$-plane as before.
The figure shows that the argument of $-h_1(r)+\sqrt{-1}h_2(r)$ varies in $(0,\pi]$.
\end{proof}
\begin{figure}[htbp]
   \centerline{\input{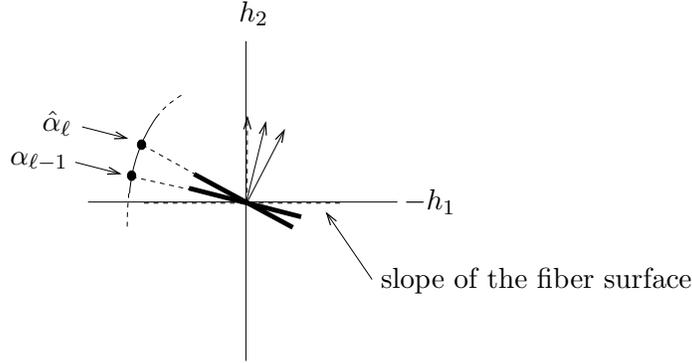}}
   \caption{Glue $\alpha_{\ell-1}$ and $\hat\alpha_\ell$ on $N(T)$ in case $m_{\ell,1}=0$.\label{fig5}}
\end{figure}

\section{Proofs of Theorem~\ref{thm01} and Corollary~\ref{cor03}}

Theorem~\ref{thm01} is included in the next assertion.

\begin{thm}\label{thm013}
Let $L(\Gamma)$ be a fibered PT graph multilink
in $S^3$. Then the following three statements are equivalent:
\begin{itemize}
\item[(1)] The compatible contact structure of $L(\Gamma)$ is tight.
\item[(2)] The compatible contact structure of each Seifert multilink before
the splicing of $L(\Gamma)$ is tight.
\item[(3)] Suppose that the denominators of the Seifert invariants
of each Seifert fibered homology $3$-sphere before the splicing of $L(\Gamma)$
are all positive. Note that we can always choose such Seifert invariants by
\cite[Proposition~7.3]{en}. In this setting, the multiplicities assigned to
the link components of $L(\Gamma)$ are either all positive or all negative.
\end{itemize}
\end{thm}

\begin{proof}
The implication (3) $\Rightarrow$ (2) is proved as follows.
If all the multiplicities of $L(\Gamma)$ are negative then we change the orientation
of $L(\Gamma)$ by the involution in~\cite[Proposition~8.1]{en} so that they are positive.
The remaining multiplicities on each Seifert multilink before the splicing are also positive
by the formula in~\cite[Corollary~10.6]{en}.
Hence the compatible contact structure of each Seifert multilink
is tight by~\cite[Theorem~1.1]{ishikawa2}.
Note that this proof works not only for PT graph multilinks in $S^3$
but also for those multilinks in any homology $3$-spheres.

We will use the condition that the manifold is $S^3$ essentially in the rest of the proof.
It is known in~\cite[Theorem~9.2]{en} that a graph link in $S^3$ is obtained from a trivial knot 
in $S^3$ by using cabling and summing operations
and the proof still works for graph multilinks in $S^3$.
Since we are considering only PT graph multilinks, we can think that 
a PT graph multilink $L(\Gamma)$ is obtained from a Seifert multilink in 
$S^3$ by iterating a splicing of a Seifert multilink in $S^3$ 
along a link component ambient isotopic to the trivial knot in $S^3$.

Let $L_1(\m_1)$ denote the initial Seifert multilink in $S^3$ and
suppose that $L(\Gamma)$ is obtained from $L_1(\m_1)$ by splicing
Seifert multilinks $L_2(\m_2),\cdots,L_{\ell'}(\m_{\ell'})$ successively.
Let $(\Sigma_{\ell-1}, L(\Gamma_{\ell-1}))$ denote the fibered PT graph multilink
obtained from $L_1(\m_1)$ by splicing
$L_2(\m_2),\cdots,L_{\ell-1}(\m_{\ell-1})$ successively, where $2\leq\ell\leq\ell'$.

We prove the implication (2) $\Rightarrow$ (1) by induction.
Without loss of generality, we assume by~\cite[Proposition~7.3]{en} that
the denominators of the Seifert invariants of each Seifert fibered homology
$3$-sphere before the splicing are all positive.
The assertion is obvious for $(\Sigma_1, L(\Gamma_1))$.
Assume that the assertion holds for $i=1,\cdots,\ell-1$ and consider the next splice
\begin{equation}\label{sp100}
    (\Sigma_{\ell}, L(\Gamma_{\ell}))=\left[(\Sigma_{\ell-1}, L(\Gamma_{\ell-1})) \textstyle\frac{\;\;\;\;\;\;}{S\;\;\;\;\;S_{\ell,1}}(\Sigma', L_\ell(\m_\ell)\right].
\end{equation}

We prepare a contact form $\hat\alpha_\ell$ whose kernel is compatible with
$L_\ell(\m_\ell)$ as follows.
The ambient space $S^3$ of $L_\ell(\m_\ell)$ is decomposed into two solid tori
$S^3\setminus \text{int\,}N(S_{\ell,1})$ and $N(S_{\ell,1})$,
where $N(S_{\ell,1})$ is the neighborhood of the link component $S_{\ell,1}$ of $L_\ell(\m_{\ell})$
for the splicing. Let $(\fM_1,\fL_1)$ and
$(\fM_2,\fL_2)$ denote the preferred meridian-longitude pairs of these
solid tori respectively. Note that they are glued in such a way that
$(\fM_1,\fL_1)=(\fL_2,\fM_2)$.
The curves in Figure~\ref{fig10a} represent a contact form $\hat\alpha_\ell$ on $S^3$
whose Reeb vector field is tangent to the fibers
of the Seifert fibration of $L_\ell(\m_\ell)$ except for small tubular neighborhoods
of its two singular fibers.
Note that the curve on the left figure represents the contact form 
on $S^3\setminus \text{int\,}N(S_{\ell,1})$ 
and the one on the right represents that on $N(S_{\ell,1})$.

\begin{figure}[htbp]
   \centerline{\input{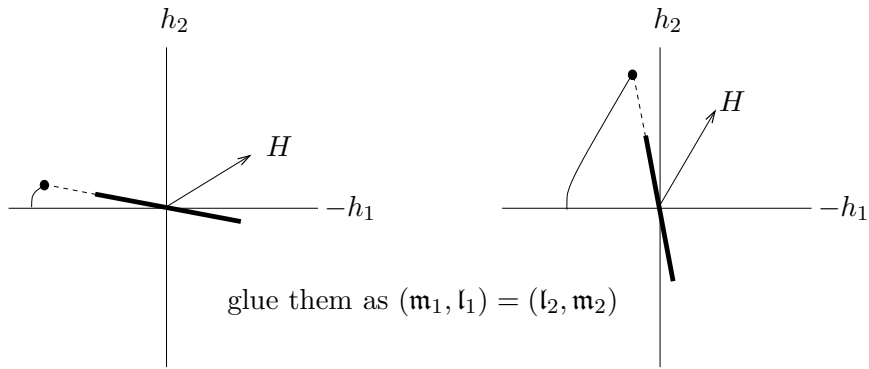}}
   \caption{A contact form on $S^3$ whose Reeb vector field is tangent to
the fibers of the Seifert fibration of $L(\m)$ except small neighborhoods of the two singular fibers.\label{fig10a}}
\end{figure}

On the other hand, by the assumption of the induction, there exists a contact form $\alpha_{\ell-1}$
whose kernel is compatible with $L(\Gamma_{\ell-1})$.
We choose a neighborhood $N(S)$ of the link component $S$
of $L(\Gamma_{\ell-1})$ for the splicing such that
$\alpha_{\ell-1}$ is given on $N(S)$ as $\alpha_{\ell-1}=c(r^2d\mu+d\lambda)$,
where $c$ is a positive constant, $(r,\mu,\lambda)$ are coordinates of 
$N(S)=D^2\times S^1$ chosen such that $(r,\mu)$ are the polar coordinates of $D^2$
with sufficiently small radius and the orientation of $\lambda$ is consistent with that of $mS$.
Note that this agrees with the working orientation of $S$ because $m>0$ by
~\cite[Theorem~1.1]{ishikawa2} and~\cite[Corollary~10.6]{en}.

The splice~\eqref{sp100} is equivalent to the replacement of
$N(S)$ in $\Sigma_{\ell-1}$ by $S^3\setminus \text{int\,}N(S_{\ell,1})$.
We now choose $S^3\setminus \text{int\,}N(S_{\ell,1})$ sufficiently small and
deform $\alpha_{\ell-1}$ such that $\hat\alpha_\ell=\alpha_{\ell-1}$
on $S^3\setminus \text{int\,}N(S_{\ell,1})=N(S)$,
which is done by applying the Gray's theorem~\cite{gray}.
We set the contact form $\alpha_\ell$ on $S^3$ to be this deformed $\alpha_{\ell-1}$.
Since the Reeb vector fields of $\alpha_{\ell-1}$ and $\hat\alpha_\ell$ are
both positively transverse to
the fiber surfaces of $L(\Gamma_\ell)$ on $\Sigma_{\ell-1}\setminus\text{int\,}N(S)$ and
$S^3\setminus\text{int\,}N(S_{\ell,1})$ respectively,
the Reeb vector field of $\alpha_\ell$ also satisfies the same property.
Hence $\ker\alpha_\ell$ is compatible with $L(\Gamma_\ell)$.

By induction, we obtain a contact form $\alpha_{\ell'}$ on $S^3$
whose kernel is compatible with $L(\Gamma)$ and contactomorphic to $\ker\alpha_1$.
Since $\ker\alpha_1$ is tight, $\ker\alpha_{\ell'}$ is also.
This completes the proof of the implication (2) $\Rightarrow$ (1).

Finally we prove the implication (1) $\Rightarrow$ (3). 
We can assume that the splice diagram is minimal since
the operations and their inverses to make equivalent splice diagrams do not change
the multiplicities at the arrowhead vertices of $\Gamma$.
If the minimal splice diagram is of type $\leftrightarrow$ then the assertion follows 
from~\cite[Theorem~1.1]{ishikawa2}.
So, we further assume that the minimal splice diagram is not of type $\leftrightarrow$.
In particular, in this setting, the fibers of the Seifert fibration intersect
the interiors of the fiber surfaces of $L(\Gamma)$ transversely
in each Seifert piece of $(S^3, L(\Gamma))$, so the diagram $\hat\Gamma$ is defined.

Let $v_\ell$ denote the inner vertex corresponding to $L_\ell(\m_\ell)$ 
and $\hat v_\ell$ the corresponding vertex in $\hat\Gamma$.
Since the graph multilink is invertible by~\cite[Theorem~8.1]{en},
we change the signs of all the multiplicities at the arrowhead vertices of $L(\Gamma)$
if necessary such that the fibers of the Seifert fibration
in $\Sigma_1$ intersect the interiors of the fiber surfaces of $L_1(\m_1)$ positively transversely.
By this assumption, we have $\hat v_1=\oplus$.
Let $\hat\Gamma_\ell$ denote the subgraph of $\hat\Gamma$ consisting of the inner vertices
$\hat v_1,\cdots,\hat v_\ell$ and the inner edges connecting them.

\begin{claim}\label{claim101}
Suppose that all vertices of $\hat\Gamma_\ell$ are $\oplus$ and all edges have sign $+$.
Suppose further that there exists an edge $e$ of $\hat\Gamma$ connected to,
but not included in, $\hat \Gamma_\ell$
with sign $-$ at the root connected to $\hat\Gamma_\ell$.
Then the compatible contact structure of $L(\Gamma)$ is overtwisted.
\end{claim}

\begin{proof}
Let $\alpha$ be a contact form on $S^3$ compatible with $L(\Gamma)$ obtained in
Proposition~\ref{lemma031}, and let $\alpha_\ell$ be a contact form 
on $\Sigma_\ell$ compatible with $L(\Gamma_\ell)$ which was obtained
in the proof of Proposition~\ref{lemma031}.
We denote by $m_1S_1,\cdots,m_\tau S_\tau$ the link components of $L(\Gamma_\ell)$
corresponding to the inner edges of $\hat\Gamma$ connected to, but not included in, $\hat \Gamma_\ell$.
From the construction, we have $\alpha=\alpha_\ell$ on 
$\Sigma_\ell\setminus\text{int\,} \sqcup_{i=1}^{\tau} N(S_i)$.
Note that, by Proposition~\ref{lemma031},
$\alpha_\ell$ has the form $\alpha_\ell=h_2(r)d\mu+h_1(r)d\lambda$ on each $N(S_i)$
such that the argument of $-h_1(r)+\sqrt{-1}h_2(r)$ varies in $(0,\pi]$.
In particular, $\ker\alpha_\ell$ does not have a half Lutz twist in $N(S_i)$.

On the other hand, let $\alpha'$ be a contact form on $\Sigma_\ell$,
whose kernel is compatible with $L(\Gamma_\ell)$,
obtained according to Proposition~\ref{prop2000}.
Remark that the edge $e$ specified in the assertion
is either an inner edge or an edge with arrowhead vertex.
Since the root of $e$ connected to $\hat \Gamma_\ell$ has sign $-$,
$(\Sigma_\ell, \ker\alpha')$ has an overtwisted disk $D$
in a tube of a half Lutz twist along the link component of $L(\Gamma_\ell)$ corresponding to $e$.
By deforming $\alpha'$ and applying the Gray's theorem~\cite{gray},
we can find a contact form $\alpha''$ on $\Sigma_\ell$ such that $\ker\alpha''$
is contactomorphic to $\ker\alpha'$, the equality $\alpha''=\alpha_\ell$ is satisfied
on $\sqcup_{i=1}^{\tau} N(S_i)$, and $\bd D$ does not intersect $\sqcup_{i=1}^{\tau} N(S_i)$
in $(\Sigma_\ell,\ker\alpha'')$.
Now we just follows the proof of~\cite[Proposition~9.2.7]{os}.
Prepare a $1$-form $\eta$ such that $d\eta$ is a volume form on the fiber surfaces
and vanishes near $\sqcup_{i=1}^{\tau} N(S_i)$, and then
show that $\ker\alpha''$ is contactomorphic to $\ker\alpha_\ell$ with preserving
the contact form $\alpha''=\alpha_\ell$ on $\sqcup_{i=1}^{\tau} N(S_i)$,
by connecting the contact forms $\alpha''+s\eta$ and $\alpha_\ell+s\eta$ with $s\geq 0$
by a one parameter family of contact forms.
Thus we conclude that the boundary $\bd D$ 
of the overtwisted disk $D$ in the contact manifold $(\Sigma_\ell, \ker\alpha_\ell)$ is 
included in $\Sigma_\ell\setminus\sqcup_{i=1}^{\tau} N(S_i)$.
This $D$ is still an overtwisted disk in $(S^3,\ker\alpha)$ 
since $\alpha=\alpha_\ell$ on $\Sigma_\ell\setminus \text{int\,}\sqcup_{i=1}^{\tau} N(S_i)$.
\end{proof}

\begin{claim}\label{claim102}
Suppose that all vertices of $\hat\Gamma_\ell$ are $\oplus$ and all edges have sign $+$.
Suppose further that all edges of $\hat\Gamma$ connected to $\hat \Gamma_\ell$
have sign $+$ at the endpoints connected to $\hat\Gamma_\ell$.
Let $e$ be an inner edge of $\hat\Gamma$ connected to, but not included in, $\hat\Gamma_\ell$.
Then the sign of the other endpoint of $e$ is $+$ and the inner vertex at
the other endpoint is $\oplus$.
\end{claim}

\begin{proof}
The first assertion follows from~\cite[Corollary~10.6]{en}.
We prove the second assertion. Let 
\[
     (\Sigma,L(\Gamma))=\left[(\Sigma_\ell,L(\Gamma_\ell)) \textstyle\frac{\;\;\;\;\;\;}{S_1\;\;\;\;\;S_2} (\Sigma', L(\Gamma'))\right]
\]
be the splice at the inner edge $e$. By the assumption and the first assertion,
the multiplicities $m_1$ of $S_1$ and $m_2$ of $S_2$ are both non-negative.
For $i=1,2$, let $N(S_i)=D^2\times S^1$ be the tubular neighborhood of $S_i$ for the splicing
and set coordinates $(r_i,\mu_i,\lambda_i)$ on $N(S_i)$ such that
$(r_i,\mu_i)$ are the polar coordinates of $D^2$ and
the orientation of $\lambda_i$ is consistent with the working orientation of $S_i$.
Let $(U_i,V_i)$ be a vector positively normal to the fiber surface
of $L(\Gamma)$ on $\bd N(S_i)$.
We assume, for a contradiction, that the inner vertex at the other endpoint is $\ominus$.

Suppose that $m_1>0$ and $m_2>0$.
Since the vertex at the root of $e$ connected to $\hat\Gamma_\ell$ is $\oplus$, we have $U_1>0$, 
see Figure~\ref{fig1f}.
On the other hand, since $m_2>0$ and the vertex at the other endpoint is $\ominus$,
we have $U_2<0$.
We now observe the identification of the positive normal vectors $(U_1,V_1)$ and $(U_2, V_2)$
after the gluing of the splice.
If $V_1>0$ then $(V_2,U_2)=K(U_1,V_1)$ with some constant $K>0$.
In particular, we have $U_2=KV_1>0$, which contradicts $U_2<0$.
If $V_1<0$ then $(V_2,U_2)=-K(U_1,V_1)$ with some constant $K>0$.
Hence $U_2=-KV_1>0$, which again contradicts $U_2<0$.

\begin{figure}[htbp]
   \centerline{\input{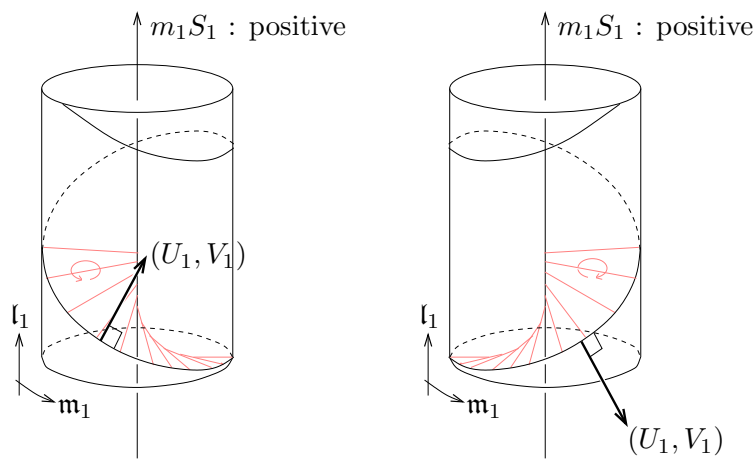}}
   \caption{The positive normal vector to the fiber surface on $\bd N(S_1)$.
Here $(\fM_1,\fL_1)$ is the preferred meridian-longitude pair on $\bd N(S_1)$
for the splicing.\label{fig1f}}
\end{figure}

If $m_1=0$ then $U_1=0$, $V_1>0$ and hence $U_2>0$, $V_2=0$ and $m_2>0$.
However $m_2>0$ implies $U_2<0$ as before, which is a contradiction.
If $m_2=0$ then $U_2=0$ and $V_2<0$ and hence $U_1<0$ and $V_1=0$,
which contradicts $m_1\geq 0$. This completes the proof.
\end{proof}

We continue the proof of Theorem~\ref{thm013}, i.e., prove the implication (1) $\Rightarrow$ (3).
By Claim~\ref{claim101}, all edges connected to $\hat v_1$ have sign $+$ at the roots
connected to $\hat v_1$, otherwise $\ker\alpha$ is overtwisted.
Then, by Claim~\ref{claim102},
the inner edge connecting $\hat v_1$ and $\hat v_2$ has sign $+$ at the endpoint
connected to $\hat v_2$ and moreover we have $\hat v_2=\oplus$.
We then use Claim~\ref{claim101} again for $\hat\Gamma_2$ and conclude that
all edges connected to $\hat \Gamma_2$ have sign $+$ at the roots connected to $\hat\Gamma_2$.
We continue this argument successively for $\hat v_3, \hat v_4,\cdots,\hat v_{\ell'}$
and finally obtain that
all inner vertices of $\hat\Gamma$ are $\oplus$ and all edges have sign $+$.
In particular, every non-empty link component of $L(\Gamma)$ has a positive multiplicity.
\end{proof}

\vspace{3mm}

\noindent
{\it Proof of Corollary~\ref{cor03}.}\;\,
The intersection $L_{fg}=\{fg=0\} \cap S_\ve$ is an oriented link in the $3$-sphere $S_\ve$ 
and the orientation is given as the boundary of the fiber surface of the 
Milnor fibration $\frac{fg}{|fg|}:S_\ve\setminus\{fg=0\}\to S^1$.
By~\cite[Appendix to Chapter~I]{en}, we know that
$L_{fg}$ is realized by a splice diagram with only $+$ signs, positive denominators
and positive multiplicities. On the other hand,
the oriented link $L_{f\bar g}=\{f\bar g=0\}\cap S_\ve$ of the fibration
$\frac{f\bar g}{|f\bar g|}:S_\ve\setminus\{fg=0\}\to S^1$
is obtained from $L_{fg}$ by reversing the orientations of
the link components corresponding to $\{g=0\}$,
as mentioned in~\cite[Proposition~3.1]{pichon}.
Hence the assertion follows from Theorem~\ref{thm01}.
\qed


\begin{thebibliography}{99} 
\bibitem{behm}K.~Baker, J.~Etnyre, J.~van Horn-Morris,
   {\it Cablings, contact structures and mapping class monoids},
   preprint, available at: arXiv.math.SG/1005.1978.
\bibitem{en}D.~Eisenbud, W.~Neumann,
   {\it Three-dimensional link theory and invariants of plane curve singularities},
   Ann. Math. Studies 110, Princeton Univ. Press, Princeton, NJ, 1985.
\bibitem{geiges}H.~Geiges,
   {\it An Introduction to Contact Topology},
   Cambridge Studies in Adv. Math. 109, Cambridge Univ. Press, 2008.
\bibitem{giroux}E.~Giroux,
   {\it G\'eom\'etrie de contact:
   de la dimension trois vers dimensions sup\'erieures},
   Proceedings of the International Congress of Mathematicians,
   Vol II (Beijing, 2002), pp.405--414,
   Higher Ed. Press, Beijing, 2002.
\bibitem{gray}J.~Gray,
   {\it Some global properties of contact structures},
   Ann. of Math. {\bf 69} (1959), 421--450.
\bibitem{ishikawa} M.~Ishikawa,
   {\it On the contact structure of a class of real analytic germs of the form $f\bar{g}$},
Singularities --- Niigata-Toyama 2007, Advanced Studies in Pure Mathematics 56, pp. 201-223, 2009.
\bibitem{ishikawa2} M.~Ishikawa,
   {\it Compatible contact structures of fibered Seifert links in homology $3$-spheres},
    preprint, available at: arXiv.math.GT/0904.0837.
\bibitem{os}B.~Ozbagci, A.I.~Stipsicz,
   {\it Surgery on Contact $3$-manifolds and Stein Surfaces},
   Bolyai Society Mathematical Studies, 13, Springer, 2004.
\bibitem{pichon}A.~Pichon,
   {\it Real analytic germs $f\bar g$ and open-book decompositions
   of the $3$-sphere},
   International J. Math. {\bf 16} (2005), 1--12.
\bibitem{ps}A.~Pichon, J.~Seade,
   {\it Fibered multilinks and real singularities $f\bar g$}, 
   Math. Ann. {\bf 324} (2008), 487--514.
\bibitem{tw}W.P.~Thurston, H.~Winkelnkemper,
   {\it On the existence of contact forms},
   Proc. Amer. Math. Soc. {\bf 52} (1975), 345--347.
\end{thebibliography}
\end{document}